\documentclass{amsart}
\usepackage{amsmath, amssymb, amscd, amsbsy, amsthm}
\usepackage{graphicx}
\usepackage{mathrsfs}
\newtheorem{theorem}{Theorem}[section]
\newtheorem{lemma}[theorem]{Lemma}
\newtheorem{prop}[theorem]{Proposition}

\theoremstyle{definition}
\newtheorem{definition}[theorem]{Definition}

\theoremstyle{remark}
\newtheorem{remark}[theorem]{Remark}
\numberwithin{equation}{section}

\def\N{{\mathbb N}}

\newcommand\T[1]{\mathscr{T}({#1})}

\begin{document}
\title{Orderings of Witzel-Zaremsky-Thompson groups}
\author{Tomohiko Ishida}
\address{Department of Mathematics, 
Kyoto University, 
Kitashirakawa Oiwake-cho, Sakyo-ku, Kyoto 606-8502, Japan.}
\email{ishidat@math.kyoto-u.ac.jp}
\subjclass[2010]{Primary 20F60, Secondary 20B07, 20B27, 20E99, 20F36}
\date{\today}

\begin{abstract}
We prove the orderability of the Witzel-Zaremsky-Thompson group 
for a direct system of orderable groups 
under a certain compatibility assumption. 
\end{abstract}
\keywords{Thompson's groups, Thompson-like groups, Witzel-Zaremsky-Thompson groups, 
braided Thompson group, orderings, braid groups, Dehornoy ordering}

\maketitle

\section{Introduction}

Thompson's groups are 
interesting countable groups  
and several kinds of their generalizations have been developed and studied. 
As one of such examples, 
Witzel and Zaremsky introduced the group 
$\T{G_{\ast}}$ associated to a direct system $(G_{n})$ of groups 
satisfying certain axioms called the cloning system \cite{wz14p}. 
We call it the Witzel-Zaremsky-Thompson group for $(G_{n})$.

The Witzel-Zaremsky-Thompson group $\T{G_{\ast}}$ 
can be considered as 
not only a generalization of Thompson's groups 
but also a limit of the groups $G_{n}$. 
We prove the group $\T{G_{\ast}}$ inherits 
the orderability of the groups $G_{n}$ 
under the assumption 
that there exists an ordering of $G_{n}$ 
which are compatible with the cloning maps. 
A precise definition of the compatibility 
is given in Section \ref{preliminaries}. 
\begin{prop}\label{main}
Let $(G_{n})$ be a cloning system. 
\begin{enumerate}
\item[(i)]
The Witzel-Zaremsky-Thompson group $\T{G_{\ast}}$ 
for $(G_{n})$ is left-orderable 
if the groups $G_{n}$ admit left-orderings 
which are compatible with the cloning maps. 
\item[(ii)]
Suppose that the cloning system is pure. 
The group $\T{G_{\ast}}$ is bi-orderable 
if the groups $G_{n}$ admit bi-orderings 
which are compatible with the cloning maps. 
\end{enumerate}
\end{prop}
The statement (ii) in Proposition \ref{main} 
is proved 
by an argument which is a straightforward generalization 
of that used in \cite{bg08}.
The author does not know
whether the assumption that the cloning system is pure 
is essential or not. 

\section{Preliminaries}\label{preliminaries}

Let $(G_{n})$ be a direct system of groups. 
That is, an injective homomorphism 
$G_{n}\hookrightarrow G_{n+1}$ is given for each $n\in\mathbb{N}$. 
The Witzel-Zaremsky-Thompson group $\T{G_{\ast}}$ is 
simply a certain subgroup of the cancellative group of the Brin-Zappa-Sz\'{e}p product 
$G\bowtie\mathscr{F}$, 
where $G$ is the direct limit of the system $(G_{n})$ 
and $\mathscr{F}$ is the monoid of binary forests. 
However we recall the definition of $\T{G_{\ast}}$ 
in a fashion describing its elements.  

Let $\mathfrak{S}_{n}$ be the symmetric group on the set $\{ 1, \dots , n\}$. 
Set the map $\varsigma _{n}^{k}\colon\mathfrak{S}_{n}\to\mathfrak{S}_{n+1}$ 
by 
\[
\varsigma_{n}^{k}(g)m
=\left\{
\begin{array}{ll}
gm & \text{if }m \leq k, gm \leq gk, \\
gm+1 & \text{if }m < k, gm > gk, \\
g(m-1) & \text{if }m > k, g(m-1) < gk, \\
g(m-1) + 1 & \text{if }m > k, g(m-1) \geq gk.
\end{array}
\right.
\]
\begin{definition}
Suppose 
that a group homomorphism 
$\rho _{n}\colon G_{n}\to\mathfrak{S}_{n}$  
and an injective map $\kappa _{n}^{k}\colon G_{n}\to G_{n+1}$ is given 
for each $n\in\mathbb{N}$ and for each $k=1, \dots , n$. 

The triple $(G_{n})=((G_{n}), (\rho _{n}), (\kappa _{n}^{k}))$ 
is a {\it cloning system} if the following three axioms are satisfied. 
\begin{enumerate}
\item
$\kappa _{n}^{k}(gh)=\kappa _{n}^{k}(g)\kappa _{n}^{\rho _{n}(g)k}(h)$ 
for $g, h\in G_{n}$, 
\item
$\kappa _{n+1}^{k}\circ\kappa _{n}^{l}
=\kappa _{n+1}^{l+1}\circ\kappa _{n}^{k}$ 
for $1\leq k<l\leq n$, and
\item
$\rho _{n}(\kappa _{n}^{k}(g))
=\varsigma _{n}^{k}(\rho _{n}(g))
\text{ or }s_{k}\varsigma _{n}^{k}(\rho _{n}(g))$ 
for $g\in G_{n}$ and for $1\leq k\leq n$, 
where the symbol $s_{k}$ means 
the transposition $(k, k+1)\in\mathfrak{S}_{n+1}$. 
\end{enumerate}
The maps $\kappa _{n}^{k}\colon G_{n}\to G_{n+1}$ 
are called the {\it cloning maps} of the cloning system $(G_{n})$. 
\end{definition}
Suppose that 
a cloning system $(G_{n})=((G_{n}), (\rho _{n}), (\kappa _{n}^{k}))$ 
is given. 
A tree diagram is a triple $(T_{-}, g, T_{+})$, 
where $g\in G_{n}$ for certain $n\in\mathbb{N}$ 
and $T_{\pm}$ are rooted planar binary trees with $n$ leaves. 
A {\it simple expansion} of $(T_{-}, g, T_{+})$ 
is a tree diagram of the form $(T'_{-}, \kappa _{n}^{k}(g), T'_{+})$. 
Here, $T'_{-}$ and $T'_{+}$ are trees 
obtained by adjoining carets to the $k$-th and $(\rho (g)k)$-th leaves 
of $T_{-}$ and $T_{+}$, respectively. 
An {\it expansion} of $(T_{-}, g, T_{+})$ 
is a tree diagram obtained as an iterated simple expansions. 
Two tree diagrams are {\it equivalent} 
if they have a common expansion. 
The equivalence class 
represented by a tree diagram $(T_{-}, g, T_{+})$ 
will be denoted by $[T_{-}, g, T_{+}]$. 
The {\it Witzel-Zaremsky-Thompson group} $\T{G_{\ast}}$ 
consists of equivalence classes of tree diagrams as a set. 
For two elements $[S_{-}, f, S_{+}], [T_{-}, g, T_{+}]\in\T{G_{\ast}}$, 
their product is defined as follows. 
There always exist tree diagrams $(S'_{-}, f', S'_{+})$ and $(T'_{-}, g', T'_{+})$
representing $[S_{-}, f, S_{+}]$ and $[T_{-}, g, T_{+}]$, respectively 
and such that $S'_{+}=T'_{-}$. 
The product $[S_{-}, f, S_{+}][T_{-}, g, T_{+}]$ 
is defined to be $[S'_{-}, f'g', T'_{+}]$. 
This is well-defined. 
The cloning system $(G_{n})$ is {\it pure} 
if the homomorphism $\rho_{n}\colon G_{n}\to\mathfrak{S}_{n}$ 
is trivial for each $n\in\N$. 

The subgroup of $\T{G_{\ast}}$ 
consisting of elements represented by the form $(T_{-}, 1, T_{+})$ 
is isomorphic to the Thompson's group $F$. 
and we denote this subgroup again by $F$. 
If we denote by $P\T{G_{\ast}}$ 
the kernel of natural projection 
$p\colon \T{G_{\ast}}\to V$, 
then the group $P\T{G_{\ast}}$ 
consists of the elements 
of the form $[T, g, T]$, 
where $g\in{\rm Ker}\rho_{n}$
and $n$ is the number of leaves of $T$. 
If the cloning system $(G_{n})$ is pure, 
then the image of the projection 
$p\colon \T{G_{\ast}}\to V$ 
is equal to Thompson's group $F<V$ 
and thus there exists an isomorphism 
$\T{G_{\ast}}\cong P\T{G_{\ast}}\rtimes F$. 

When we fix a left-ordering or a bi-ordering $<_{n}$ 
of $G_{n}$ for each $n$, 
we say the orderings are {\it compatible with the cloning maps} $\kappa _{n}^{k}$'s 
if $1<_{n+1}\kappa _{n}^{k}(g)$ for any $n$ and any $k$ 
whenever $1<_{n}g$. 

\section{Proof of the main proposition}\label{proof}
\begin{proof}[Proof of Proposition \ref{main}]
Let $(G_{n})=((G_{n}),  (\rho _{n}), (\kappa _{n}^{k}))$ 
be a cloning system. 
Suppose that a left-ordering $<_{n}$ of the group $G_{n}$ is given for each $n$ 
and the orderings $<_{n}$ are compatible with the cloning maps. 
Since the Thompson's group $F$ is known to be bi-orderable,  
we fix a left-ordering $<_{F}$ of $F$. 
Set the subset $\Pi$ 
of the Witzel-Zaremsky-Thompson group $\T{G_{\ast}}$ for $(G_{n})$ by 
\[ \Pi =\{ [T_{-}, g , T_{+}]\in\T{G_{\ast}}; 
g\in G_{n}, 1<_{n} g\} 
\amalg\{f\in F; 1<_{F} f\}. \]
Since the orders $<_{n}$ are compatible with the cloning maps, 
the set $\Pi$ is well-defined. 
Further it is easy to verify 
that $\Pi$ is a positive cone 
and thus defines a left-ordering of $\T{G_{\ast}}$. 
This completes the proof of the statement (i). 

If the groups $G_{n}$ admit bi-orderings 
which are compatible with the cloning maps, 
then the bi-ordering on $P\T{G_{\ast}}$ 
is induced. 
Under the assumption that the cloning system $(G_{n})$ is pure, 
since both the groups $P\T{G_{\ast}}$ and $F$ 
are bi-orderable, 
the group $\T{G_{\ast}}$ 
which is isomorphic to 
the semi-direct product $P\T{G_{\ast}}\rtimes F$ 
is also. 
This completes the proof of the statement (ii). 
\end{proof}
\begin{remark}
If we fix a left-ordering or bi-ordering 
on the group $\T{G_{\ast}}$, 
its restriction on $P\T{G_{\ast}}$ 
uniquely determines a family of left-ordering or bi-ordering  respectively 
on the groups $G_{n}$. 
In fact, 
if for each rooted planar binary tree $T$ with $n$ leaves, 
we denote by $G_{T}$ the subgroup of $P\T{G_{\ast}}$ 
consisting of elements represented by tree diagrams of the form $(T, g, T)$, 
where $g\in G_{n}$, 
then $G_{n}$ is isomorphic to $G_{n}$. 
Furthermore, if $T$ and $T'$ are rooted planar binary trees 
with same number of leaves, 
then in $\T{G_{\ast}}$ 
the subgroups group $G_{T}$ and $G'_{T}$ are conjugate to each other 
by an element of $F$. 
Thus if we assume the cloning system $(G_{n})$ is pure, 
the converse of each statement of Proposition \ref{main} 
always holds. 
\end{remark}

\section{Examples}

In this section we discuss examples of Witzel-Zaremsky-Thompson groups 
to which Proposition \ref{main} is applicable. 

\subsection{The braided Thompson group}
Let $B_{n}$ be the braid group of $n$ strands 
and $\sigma _{1}, \dots , \sigma _{n-1}\in B_{n}$ 
the Artin generators.  
Taking $\rho _{n}\colon B_{n}\to\mathfrak{S}_{n}$ 
to be the natural projection 
and setting $\kappa _{n}^{k}\colon B_{n}\to B_{n+1}$ by
\[ \kappa _{n}^{k}(\sigma _{i})
=\left\{
\begin{array}{ll}
\sigma _{i+1} & \text{if }k<i, \\
\sigma _{i+1}\sigma _{i} & \text{if }k=i, \\
\sigma _{i}\sigma _{i+1} & \text{if }k=i+1, \\
\sigma _{i} & \text{if }i+1<k,
\end{array}
\right. \] 
we have a well-defined cloning system on $(B_{n})$. 
Then the Witzel-Zaremsky-Thompson group 
$\T{B_{n}}$ for $(B_{n})$ 
is isomorphic to the group 
called the braided Thompson group $BV$
which was independently introduced by Brin and Dehornoy 
\cite{brin07}\cite{dehornoy06}. 

Recall that $\beta\in B_{n}$ is positive 
with respect to the Dehornoy ordering 
if and only if 
$\beta$ is represented by a word $w$ in the Artin generators 
which satisfies the following condition (D): 

\noindent  \textbf{(D)} 
$\sigma _{i}$ occurs in $w$ 
but $\sigma _{1}^{\pm 1}, \dots , \sigma _{i-1}^{\pm 1}, \sigma _{i}^{-1}$ does not 
for certain $i$. 

It is not difficult to verify only by definition of the cloning maps 
that if $w$ is a representation of $\beta$ satisfying the condition (D) 
then $\kappa _{n}^{k}(w)$ also. 
Hence the following lemma holds:
\begin{lemma}\label{Bn}
The Dehornoy orderings of the braid groups $B_{n}$
are compatible with the cloning maps of $(B_{n})$ defined above. 
\end{lemma}
By Proposition \ref{main} and Lemma \ref{Bn}, 
we have in easy way 
the following theorem which was first proved in \cite{dehornoy06}. 
\begin{theorem}
The braided Thompson group $BV$ is left-orderable. 
\end{theorem}

\subsection{The pure braided Thompson group}

Let $P_{n}$ be the pure braid group of $n$ strands.
We also denote the restriction on $P_{n}$ 
of the map $\kappa _{n}^{k}\colon B_{n}\to\mathfrak{S}_{n}$ 
we set in the previous subsection again by $\kappa _{n}^{k}$. 
Take $\rho _{n}\colon P_{n}\to\mathfrak{S}_{n}$ 
to be the trivial homomorphism. 
Then we have a well-defined cloning system on $(P_{n})$ 
and the Witzel-Zaremsky-Thompson group $\T{P_{n}}$ for $(P_{n})$
which is isomorphic to the pure braided Thompson group $BF$ 
introduced by Brady, Burillo, Cleary and Stein in \cite{bbcs08}. 
\begin{lemma}[\cite{bg08}]
The orderings on the braid groups $P_{n}$
induced from the Magnus orderings of the free groups 
via the Artin combing 
are compatible with  the cloning maps on $(P_{n})$ defined above. 
\end{lemma}
\begin{theorem}[\cite{bg08}]
The pure braided Thompson group $BF$ is bi-orderable. 
\end{theorem}

\subsection{Witzel-Zaremsky-Thompson group for direct powers of a group}

For arbitrary group $G$, 
set $G_{n}$ to be the $n$-th direct power $G^{n}$ of $G$. 
Fix injective homomorphisms 
$\phi _{1}, \phi _{2}$ of $G$. 
If we define 
$\kappa _{n}^{k}\colon G^{n}\to G^{n+1}$ by 
$\kappa _{n}^{k} (g_{1}, \dots , g_{n})
=(g_{1}, \dots , \phi _{1}(g_{k}), \phi _{2}(g_{k}), \dots , g_{n})$ 
and set $\rho _{n}\colon G^{n}\to\mathfrak{S}_{n}$ 
to be the trivial homomorphism, 
then we have a cloning system on direct powers $(G^{n})$ of the group $G$ 
and the Witzel-Zaremsky-Thompson group $\T{G^{n}}$ 
which was introduced by Tanushevski \cite{tanushevski16}. 
Suppose that $G$ is left-orderable 
and fix a left-ordering $<$ of $G$. 
If the homomorphism $\phi _{1}\colon G\to G$ 
preserves the order of $G$, 
then the lexicographic orderings of $G^{n}$'s 
induced by $<$ 
are compatible with the cloning maps of $(G^{n})$. 
Further if the order on $G$ is bi-invariant, 
then the induced orders on $G^{n}$ are also. 
Since the identity map of $G$ 
trivially preserves all the orderings of $G$, 
by Proposition \ref{main} we have the following Theorem: 
\begin{theorem}
If the group $G$ is left-orderable or bi-orderable, 
then the Witzel-Zaremsky-Thompson group $\T{G^{n}}$ 
of the direct powers of $G$ 
is also left-orderable or bi-orderable, respectively. 
\end{theorem}

\vskip 5pt
\noindent \textbf{Acknowledgments.}
The author is grateful to the referee 
for carefully reading the manuscript and 
pointing out mistake in it. 
The author was supported by JSPS Research Fellowships
for Young Scientists (26$\cdot$110). 

\providecommand{\bysame}{\leavevmode\hbox to3em{\hrulefill}\thinspace}
\providecommand{\MR}{\relax\ifhmode\unskip\space\fi MR }
\providecommand{\MRhref}[2]{%
  \href{http://www.ams.org/mathscinet-getitem?mr=#1}{#2}
}
\providecommand{\href}[2]{#2}

\end{document}